\documentclass[a4paper]{amsart}
\usepackage{amsmath}
\usepackage{amsfonts}
\usepackage{amssymb}
\usepackage{mathrsfs}
\usepackage[mathcal]{eucal}
\usepackage{enumerate}
\usepackage{chngcntr}
\counterwithin{equation}{section}

\theoremstyle{plain}
\newtheorem{thm}{Theorem}[section]
\newtheorem{lem}[thm]{Lemma}
\newtheorem{prop}[thm]{Proposition}
\newtheorem{cor}[thm]{Corollary}

\theoremstyle{definition}
\newtheorem{defn}{Definition}[section]

\newtheorem{exmp}{Example}[section]

\theoremstyle{remark}
\newtheorem*{rem}{Remark}

\DeclareMathOperator*{\Perm}{\mathit{Perm}}
\DeclareMathOperator*{\Aut}{\mathit{Aut}}

\begin{document}
\title[Skew two-sided bracoids]{Skew two-sided bracoids}
\author{Izabela Agata Malinowska}

\address{Institute of Mathematics, University of Bia\l{}ystok,  
 15-245 Bia\l{}ystok, Cio\l{}kowskiego 1\,M, Poland}
 
\email{izabelam@math.uwb.edu.pl}
\date{}


      
\date{}

\begin{abstract}
Isabel Martin-Lyons and Paul J.Truman generalized the definition of a skew brace to give a new algebraic object, which they termed a skew bracoid.
 Their construction involves two groups interacting in a manner analogous to the compatibility condition found in the definition of a skew brace.
 They formulated tools for characterizing and classifying skew bracoids, and studied substructures and quotients of skew bracoids. In this paper
 we study two-sided bracoids. In \cite{WR07} Rump showed that if a left brace 
$(B, \star ,\cdot )$ is  a two-sided brace and the operation $\ast : B \times B \longrightarrow  B$ is
 defined by
$a \ast b = a\cdot b \star \overline{a} \star \overline{b}$ for all $a, b \in B$ then $(B, \star ,\ast )$ is a Jacobson radical ring.  Lau showed that if 
$(B, \star ,\cdot )$ is a left brace and the operation  is asssociative, then 
$B$ is a two-sided brace. We will prove   bracoid versions of this results. 
\end{abstract}. 
\maketitle

\section{Introduction and notation}

Skew braces has been introduced by Guarnieri and Vedramin in \cite{GV17} as a generalisation of the left
braces introduced by Rump in \cite{WR07}. 

\begin{defn}\label{defn21}
A \emph{ (left) skew brace} is a triple $(B, \star , \cdot   )$ such that $(B, \cdot )$ and $(B, \star  )$ are  groups  such that
\begin{equation} \label{eq1}
a \cdot  (b \star   c) = (a \cdot  b) \star   \overline{a} \star   (a \cdot c) 
\end{equation}
for all $a,b,c\in  B$.
\end{defn}

Throughout, the identity element of a group $G$ is denoted $e_G$.
We also denote by $\overline{\eta}$ the inverse element of $\eta$ in a group $(N, \star )$ and $g^{-1}$ the inverse element of $g$ in a group $(G,\cdot )$.
Note that if $(B,\star )$ is an abelian group, we obtain the definition of left brace introduced by Rump, as formulated by Ced\'{o}, Jespers and Okni\'{n}ski in
\cite{CJO14}.
Skew left braces have been devised with the aim of attacking
the problem of finding all set-theoretical non-degenerate solutions 
of the Yang-Baxter Equation, a consistency equation that plays a relevant role
in quantum statistical mechanics, in the foundation of quantum groups, and that provides a multidisciplinary approach from a wide variety of areas such as Hopf algebras,
knot theory and braid theory among others.

In \cite{BT19}  Brzezi\'{n}ski proposed to study a set with two binary operations connected by a rule  which can be seen as the interpolation between
 the ring-type (i.e. the standard) and brace distributive laws. A \emph{skew left truss} is a set $A$ with binary operations $\star $ and $\cdot$, such that $(A, \star )$ is a group,
$ (A, \cdot )$ is a semigroup and, for all
$a, b, c \in A$, the following generalised distributive law holds
\begin{equation*} 
a \cdot  (b \star   c) = (a \cdot  b) \star   \overline{a\cdot e_\star } \star   (a \cdot c), 
\end{equation*}
where $e_\star $ is the neutral element of the group $(A, \star )$.
This truss distributive law interpolates the ring (standard) and brace distributive laws:
the former one is obtained by setting $a\cdot e_\star =  e_\star $, the
latter is obtained by setting $a\cdot e_\star = a$. A special case of  skew left truss is a near brace, considered by 
Doikou and Rybo\l owicz in \cite{DR24}. A near brace is a skew left truss  $A$, where $ (A, \cdot )$ is a group. Using the notion of the near braces 
they produced new multi-parametric,
non-degenerate, non-involutive solutions of the set-theoretic Yang-Baxter equation. These
solutions are generalisations of the known ones coming from braces and skew braces. 

In \cite{MT24}   Martin-Lyons and Truman generalised the definition of a skew brace to give a new algebraic object, which they termed 
a skew bracoid. 

\begin{defn}\label{defn22} \cite[Definition 2.1.]{MT24}
A \emph{ (left) skew bracoid} is a $5$-tuple $(G, \cdot ,N, \star   ,\odot )$ such that $(G, \cdot )$ and $(N, \star )$ are groups and $\odot $ is a transitive action of $(G, \cdot )$ on $N$ such that
\begin{equation} \label{eq2}
g \odot  (\mu \star   \eta) = (g \odot  \mu) \star   \overline{(g \odot  e_N)}  \star  (g \odot \eta) 
\end{equation}
for all $g\in G$ and $\mu,\eta \in N$. If  $(N, \star )$ is an abelian group, then we shall call $(G, \cdot ,N, \star   ,\odot )$ 
a \emph{ (left) bracoid}. 
\end{defn}

\begin{exmp}\label{e21} \cite[Example 2.2.]{MT24}
A skew brace $(B, \star , \cdot   )$ can be viewed as a skew bracoid with $(G, \cdot )=(B, \cdot )$,  $(N, \star )=(B, \star )$ and 
$\odot = \cdot $.
\end{exmp}

Clearly a near brace is also a special case of a skew bracoid. 
Martin-Lyons and Truman showed that several important constructions and identities associated with skew braces (such as $\gamma $-functions)
 have natural skew bracoid counterparts. We recall  that the
 $\gamma $-function of a skew brace $(B, \cdot , \star    )$ is the function $\gamma : B \longrightarrow \Perm(B)$ defined by 
$^{\gamma (b)}a =\overline{b} \star  (b\cdot a)$. In fact, we have $\gamma (b) \in \Aut(B, \star )$ for each $b \in B$, and the map 
$\gamma $ is a homomorphism from
 $(B, \cdot )$ to $\Aut(B, \star )$.  The $\gamma $-functions of skew braces (often called $\lambda$-functions) 
are used, among others, to define left ideals and ideals of skew braces (see for example \cite{CSV19, GV17}).  In \cite{MT24}  Martin-Lyons
and Truman showed that certain skew bracoids arise from a natural quotienting procedure on skew braces. 
They also defined $\gamma$-functions, left ideals and ideals of skew braces.

\begin{defn}\label{defn23a} {\cite[Definition 2.10.]{MT24}} Let $(G, \cdot ,N, \star   ,\odot )$ be a skew bracoid. The homomorphism 
$\gamma: G \longrightarrow \Aut(N)$ defined by 
$$^{\gamma(g)}\eta = \overline{(g \odot  e_N)} \star  (g \odot \eta)\, \text{for\, all}\, g \in G\, \text{and}\, \eta\in N$$
is called the \emph{$\gamma $-function} of the skew bracoid.
\end{defn}

As it was noticed in \cite{MT24} when a skew brace is viewed as a skew bracoid, the $\gamma$-function of the skew bracoid coincides with 
that of the skew brace.

In Section 3 we consider  two-sided bracoids. In \cite{WR07} Rump showed that if a left brace 
$(B, \star ,\cdot )$ is  a two-sided brace and the operation $\ast : B \times B \longrightarrow  B$ is
 defined by
$a \ast b = a\cdot b \star \overline{a} \star \overline{b}$ for all $a, b \in B$ then $(B, \star ,\ast )$ is a Jacobson radical ring
 (see also \cite[Proposition 1]{CGS18}). In \cite{L} Lau showed that if $(B, \star ,\cdot )$ is a left brace and the operation $\ast$  is asssociative, then 
$B$ is a two-sided brace. We will prove   bracoid versions of this results.

\section{Preliminaries}

Here are some elementary properties of skew bracoids.

\begin{rem} {\rm (see \cite[Remark 1.3.]{CSV19})}
 Let $(G, \cdot ,N, \star   ,\odot )$ be a skew left bracoid. Then the following formulas hold:
$$(g \odot  e_N) \star  ^{\gamma(g)}\eta = g \odot \eta,\; (g \odot  e_N) \star  \eta = g \odot (^{\gamma(g)^{-1}}\eta),\; ^{\gamma(g)}(g^{-1} \odot e_N) =
 \overline{ g \odot e_N}$$
 for all $g\in G$ and $\eta\in N$.
\end{rem}

\begin{lem}\label{lem23} {\cite[Proposition 2.13.]{MT24}} Let $(G, \cdot ,N, \star   ,\odot )$ be a skew bracoid. Then for all $g\in G$ and $\eta\in N$ we have
$$\overline{(g \odot  e_N)} \star (g \odot \overline{ \eta}) \star \overline{ (g \odot  e_N)} = \overline{g \odot  \eta}.$$
\end{lem}

Now we generalise the $\ast $-operation of a skew brace $(B,\star ,\cdot)$.
Recall that the operation $\ast : B \times B \longrightarrow  B$ is
 defined by setting
$$a \ast b = \overline{a} \star (a\cdot b) \star  \overline{b}$$ for all $a, b \in B$.

\begin{defn}
Let $(G, \cdot ,N, \star   ,\odot )$ be a skew bracoid. For every $g \in G$ define the map
 $\alpha (g) : (N, \star ) \longrightarrow (N,\star )$   by setting
\begin{equation} \label{eq3}
^{\alpha (g)}\eta = \overline{(g \odot  e_N) } \star  (g \odot \eta) \star \overline{ \eta} = (^{\gamma(g)}\eta) \star \overline{ \eta},
\end{equation}
for all  $\eta \in N$.
\end{defn}

When a skew brace $(B,\star ,\cdot)$ is viewed as a skew bracoid we have $^{\alpha (a)}b = a\ast b$ for all $a, b \in B$ (see \cite{CSV19}).

We begin with some essential properties of the map $\alpha (g) $.

\begin{lem}\label{lem24} Let $(G, \cdot ,N, \star   ,\odot )$ be a skew bracoid. Then for every $g, h\in G$ and $\mu, \eta \in N$, we have:
\begin{enumerate}[\rm(1)]
\item $^{\alpha (g)}(\eta \star  \mu) = (^{\alpha (g)}\eta) \star  \eta \star  (^{\alpha (g)}\mu) \star \overline{ \eta}.$
\item $^{\alpha (g)}e_N = ^{\alpha (e_G)}  \eta =e_N.$
\item $^{\alpha (g)}\overline{\eta}  = \overline{\eta} \star (\overline{ ^{\alpha (g)}\eta}) \star \eta.$ 
\item $^{\alpha (gh)}\eta = (^{\alpha (g)}(^{\alpha (h)}\eta)) \star  (^{\alpha (h)}\eta) \star  (^{\alpha (g)}\eta).$
\end{enumerate}
\end{lem}
\begin{proof} Only Statement (4) is in doubt.
\begin{multline*} 
(^{\alpha (g)}(^{\alpha (h)}\eta)) \star  (^{\alpha (h)}\eta) \star  (^{\alpha (g)}\eta) =
 (^{\gamma(g)}(^{\alpha (h)}\eta)) \star (\overline{ ^{\alpha (h)}\eta }) \star  (^{\alpha (h)}\eta) \star  (^{\alpha (g)}\eta) = \\ 
(^{\gamma(g)}(^{\gamma(h)}\eta \star \overline{ \eta}) )\star  (^{\gamma(g)}\eta) \star \overline{ \eta} = 
(^{\gamma(g)}(^{\gamma(h)}\eta \star \overline{ \eta} \star  \eta)) \star \overline{ \eta} = (^{\gamma(gh)}\eta) \star \overline{ \eta} = ^{\alpha (gh)}  \eta .
\end{multline*} 
\end{proof}

\begin{cor}\label{c23} Let $(G, \cdot ,N, \star   ,\odot )$ be a  bracoid (i.e. $(N,\star )$ is an abelian group).  Then for every $g \in G$ the map
$\alpha (g)$ is an endomorhism of  the group $(N,\star )$.

\end{cor}

A right skrew bracoid is defined analogously to a left skew bracoid.

\begin{defn}\label{defn31}
A \emph{ right skew bracoid} is a $5$-tuple $(H, \circ ,N, \star   ,\boxdot )$ such that $(H, \circ )$ and $(N, \star )$ are groups and $\boxdot $ 
is a transitive right action of $(H, \circ )$ on $(N,\star )$ such that
\begin{equation} \label{eq4}
  (\eta \star   \mu) \boxdot  g = (  \eta \boxdot g) \star   \overline{(  e_N \boxdot g)}  \star  ( \mu \boxdot g) 
\end{equation}
for all $g\in H$ and $\eta,\mu \in N$. If $(N,\star )$ is an abelian group, then we shall call $(H, \circ ,N, \star   ,\boxdot )$ a \emph{ right bracoid}.
\end{defn}

Let $(H, \circ ,N, \star   ,\boxdot )$ be a right skew bracoid. We can  define the homomorphism $\delta : H \longrightarrow \Aut(N)$ as 
$$\eta^{\delta (h)}  = (\eta \boxdot h) \star  \overline{(e_N \boxdot h)}\, \text{for\, all}\, h \in H\, \text{and}\, \eta\in N.$$
We call it the \emph{$\delta $-function} of the right skew bracoid. 

\begin{defn} \label{defn23} Let $(H, \circ )$ and $(N, \star )$ be groups and $\boxdot $ 
be a transitive right action of $(H, \circ )$ on $(N,\star )$. We define for every $h\in H$ the map $\beta (h) : (N, \star ) \longrightarrow (N, \star )$
by setting 
\begin{equation} \label{eq5}
  \eta ^{\beta (h)} = \overline{\eta  } \star  (\eta \boxdot h) \star \overline{ (e_N  \boxdot h)},
\end{equation}
for all $\eta \in N$.
\end{defn}

Let $(H, \circ ,N, \star   ,\boxdot )$ be a right skew bracoid. Then 
$$\eta ^{\beta (h)} =  \overline{ \eta}  \star \eta^{\delta(h)}$$
for all $h\in H$ and $\eta \in N$. 
We can also prove some essential properties of the map $\beta (h)$ for every $h\in H$.

\begin{lem}\label{lem31} Let $(H, \circ ,N, \star   ,\boxdot )$ be a right skew bracoid.  Then for every $ h\in H$ and $\mu, \eta \in N$, we have:
\begin{enumerate}[\rm(1)]
\item $ (\mu \star  \eta)^{\beta (h)} =   \overline{ \eta} \star ( \mu ^{\beta (h)}) \star  \eta   \star  (\eta ^{\beta (h)}).$
\item $ e_N ^{\beta (h)} =  \eta ^{\beta (e_H)} =e_N.$
\item $ \overline{\eta}^{\beta (h)}  = \eta \star (\overline{  \eta ^{\beta (h)}}) \star \overline{\eta }.$ 
\end{enumerate}
\end{lem}

\begin{cor}\label{c25} Let $(H, \circ ,N, \star   ,\boxdot )$ be a right  bracoid. Then for every  $h\in H$ the map
$\beta (h)$ is an endomorhism of  the  group $(N,\star )$.
\end{cor}

When a skew brace $(B,\star ,\cdot)$ is viewed as a skew bracoid we have $a^{\beta (b)} = a\ast b$ for all $a, b \in B$.

\section{Two-sided bracoids}

\begin{defn}\label{defn32}
Let $(G,\cdot )$, $(H, \circ )$ and $(N, \star )$ be groups.
If $(G, \cdot ,N, \star   ,\odot )$ is a left skrew bracoid and $(H, \circ ,N, \star   ,\boxdot )$ is a right skew bracoid  such that
\begin{equation} \label{eq6}
 g \odot  ( \eta \boxdot  h) = (g \odot   \eta) \boxdot  h 
\end{equation}
for all  $ g\in G$, $h\in H$ and $\eta \in N$,
then we shall call $(G, \cdot ,\odot , H, \circ ,\boxdot , N, \star )$ a \emph{two-sided skew bracoid}. If $(N,\star )$ is an abelian group, then we shall call 
$(G, \cdot ,\odot , H, \circ ,\boxdot , N, \star )$ a \emph{two-sided bracoid}.
\end{defn}

\begin{exmp} In this example we denote by $g^{-1}$ the inverse element of g of a group.
Let $t, w\in \mathbb{N}$, let $d$ be a positive divisor of $\gcd(t, w)$, let
\begin{equation*}
G = \langle x, y \mid x^t = y^4 = 1, x^y = x^{-1} \rangle, \quad H = \langle a, b \mid a^{2w}  = 1, a^w =b^2, a^b = a^{-1} \rangle
\end{equation*}
and let 
\begin{equation*}
N = \langle \mu, \eta \mid \mu^d = \eta^2 = 1, \mu^\eta = \mu^{-1} \rangle = D_d.
\end{equation*}
Then the rule
\begin{equation*}
x^iy^j \odot \mu^r\eta^s = \mu^{i+(-1)^jr}\eta^{j+s}
\end{equation*}
defines a transitive action of $G$ on $N,$ and we have
\begin{multline*}   (x^iy^j \odot \mu^{r_1}\eta^{s_1}) \star  (x^iy^j \odot e_N)^{-1} \star (x^iy^j \odot \mu^{r_2}\eta^{s_2}) = \\
\mu^{i+(-1)^jr_1}\eta^{j+s_1} \star (\mu^{i}\eta^{j})^{-1} \star \mu^{i+(-1)^jr_2}\eta^{j+s_2}=
\mu^{i+(-1)^jr_1+(-1)^{j+s_1}r_2}\eta^{j+s_1+s_2},
\end{multline*}  
\begin{equation*}   (\mu^{r_1}\eta^{s_1}) \star   (\mu^{r_2}\eta^{s_2}) = \mu^{r_1+(-1)^{s_1}r_2}\eta^{s_1+s_2}
\end{equation*}  
\begin{multline*}   x^iy^j \odot ((\mu^{r_1}\eta^{s_1}) \star  ( \mu^{r_2}\eta^{s_2}) )= \\
x^iy^j \odot (\mu^{r_1+(-1)^{s_1}r_2}\eta^{s_1+s_2} )=
\mu^{i+(-1)^jr_1+(-1)^{j+s_1}r_2}\eta^{j+s_1+s_2},
\end{multline*}  
Therefore $(G, \cdot  , N, \star ,\odot )$ is a left skew bracoid.

The rule
\begin{equation*}
 \mu^r\eta^s  \boxdot a^kb^l = \mu^{r+(-1)^sk}\eta^{s+l}
\end{equation*}
defines a transitive right action of $H$ on $N,$ and we have
\begin{multline*}    ((\mu^{r_1}\eta^{s_1}) \star  ( \mu^{r_2}\eta^{s_2})) \boxdot a^kb^l = \\
(\mu^{r_1+(-1)^{s_1}r_2}\eta^{s_1+s_2}) \boxdot a^kb^l =
\mu^{r_1+(-1)^{s_1}r_2 + (-1)^{s_1+s_2}k}\eta^{l+s_1+s_2},
\end{multline*}  
\begin{multline*}    ((\mu^{r_1}\eta^{s_1}) \boxdot a^kb^l ) \star (e_N \boxdot a^kb^l)^{-1} \star  (( \mu^{r_2}\eta^{s_2}) \boxdot a^kb^l) = \\
\mu^{r_1+(-1)^{s_1}k}\eta^{s_1+l} \star (\mu^k\eta^l)^{-1} \star \mu^{r_2+(-1)^{s_2}k}\eta^{s_2+l}=
\mu^{r_1+(-1)^{s_1}r_2 + (-1)^{s_1+s_2}k}\eta^{l+s_1+s_2}.
\end{multline*}  
Hence $(H, \circ  , N, \star ,\boxdot )$ is a right skew bracoid.
Finally,
\begin{equation*}   
(x^iy^j \odot \mu^r\eta^s) \boxdot a^kb^l = \mu^{i+(-1)^jr}\eta^{j+s} \boxdot a^kb^l =\mu^{i+(-1)^jr+(-1)^{j+s}k}\eta^{j+s+l}
\end{equation*}  
and
\begin{equation*}   
x^iy^j \odot (\mu^r\eta^s \boxdot a^kb^l) = x^iy^j \odot (\mu^{r+(-1)^sk}\eta^{s+l}) =\mu^{i+(-1)^jr+(-1)^{j+s}k}\eta^{j+s+l}
\end{equation*}  
Therefore $(G, \cdot ,\odot , H, \circ ,\boxdot , N, \star )$ is a skew two-sided  bracoid. 
\end{exmp}

In \cite{WR07} Rump showed that if a left brace $(B, \star ,\cdot )$ is  a two-sided brace and the operation $\ast : B \times B \longrightarrow  B$ is
 defined by
$a \ast b = a\cdot b \star \overline{a} \star \overline{b}$ for all $a, b \in B$ then $(B, \star ,\ast )$ is a Jacobson radical ring
 (see also \cite[Proposition 1]{CGS18}). Now we will prove a  bracoid version of this result.

\begin{thm}\label{thm32} If $(G, \cdot ,\odot , H, \circ ,\boxdot , N, \star )$ is a two-sided  bracoid (i.e. $(N, \star )$ is an abelian group), then for all 
$ g\in G$, $h\in H$ and $\mu , \eta \in N$ we have 
\begin{enumerate}[\rm(1)]
\item $\alpha (g)$ is an endomorphism of $(N,\star )$.
\item $\beta (h)$ is an endomorphism of $(N,\star )$.
\item $(^{\gamma (g)} \eta)^{\delta (h)} =  {^{\gamma (g)}(\eta^{\delta (h)})}$. 
\item $(^{\alpha (g)}\eta)^{\beta (h)} =   {^{\alpha (g)}(\eta ^{\beta (h)})}$.
\end{enumerate}
\end{thm}
\begin{proof}  Since $(N,\star )$ is abelian, by Corollaries \ref{c23} and \ref{c25} we obtain (1)-(2).

Now we prove (3). Let $ g\in G$, $h\in H$ and $\eta \in N$.
\begin{multline*}  
(^{\gamma (g)}\eta)^{\delta (h)} = (\overline{(g \odot  e_N)} \star  (g \odot \eta))^{\delta (h)}  = \overline{(g \odot  e_N)^{\delta (h)}} \star  (g \odot \eta)^{\delta (h)} = \\
 \overline{((g \odot  e_N)\boxdot h) \star \overline{( e_N \boxdot h)}} \star  ((g \odot  \eta)\boxdot h) \star \overline{( e_N \boxdot h)} =\\
 \overline{((g \odot  e_N)\boxdot h)} \star ( e_N \boxdot h) \star  ((g \odot  \eta)\boxdot h) \star \overline{( e_N \boxdot h)} =\\
 \overline{((g \odot  e_N)\boxdot h)}  \star  ((g \odot  \eta)\boxdot h)  =\\
  \overline{(g \odot  (e_N\boxdot h))} \star ( g\odot e_N) \star  (g \odot  (\eta\boxdot h)) \star \overline{( g\odot e_N)} =\\
\overline{^{\gamma (g)}(e_N\boxdot h)} \star {^{\gamma (g)}(\eta\boxdot h)} = {^{\gamma (g)}(\overline{(e_N\boxdot h)} \star (\eta\boxdot h))}
 = {^{\gamma (g)}(\eta^{\delta (h)})}.
\end{multline*} 
So only (4) is in doubt. Let $ g\in G$, $h\in H$ and $\eta \in N$.
Then by using (3) we have 
\begin{multline*}
(^{\alpha (g)}\eta)^{\beta (h)} =  (^{\gamma (g)}\eta \star  \overline{\eta}) ^{\beta (h)}  = 
\overline{(^{\gamma (g)}\eta \star  \overline{\eta})} \star (^{\gamma (g)}\eta \star  \overline{\eta})^{\delta (h)} = \\
 (^{\gamma (g)}\overline{\eta} \star  \eta) \star (^{\gamma (g)}\eta)^{\delta (h)} \star  \overline{\eta}^{\delta (h)} =
 {^{\gamma (g)}\overline{\eta}} \star  \eta \star {^{\gamma (g)}(\eta^{\delta (h)})} \star  \overline{\eta}^{\delta (h)} =\\
 ^{\gamma (g)}(\overline{\eta} \star  \eta^{\delta (h)}) \star \overline{(\overline{\eta} \star  \eta^{\delta (h)})}  = 
  {^{\alpha (g)}(\overline{\eta} \star  \eta^{\delta (h)})} = {^{\alpha (g)}(\eta ^{\beta (h)})}.
\end{multline*}
\end{proof}

When a two-sided brace $(B,\star ,\cdot)$ is viewed as a two-sided bracoid we have $$a^{\beta (b)} = a\ast b = {^{\alpha (a)}b}$$ for all $a, b \in B$.
Hence Theorem \ref{thm32} is a generalisation of the result of Rump.

In \cite{L} Lau showed that if $(B, \star ,\cdot )$ is a left brace and the operation $\ast : B \times B \longrightarrow  B$ defined by
$a \ast b = a\cdot b \star \overline{a} \star \overline{b}$ for all $a, b \in B$ is asssociative, then $B$ is a two-sided brace. In the next part of the section 
we will prove a  bracoid version of this result.

\begin{prop}\label{p34}  Let $(G, \cdot ,N, \star   ,\odot )$ be a left skew bracoid and $(H, \circ )$ be a group. Assume that
$\boxdot $ is a transitive right action of $(H, \circ )$ on $N$  such that
\begin{equation*} 
  g \odot  ( \eta \boxdot  h) = (g \odot   \eta) \boxdot  h  
\end{equation*}
for all $g\in G, h \in H,$ $\eta \in N$. 
Then for all  $g\in G,$ $\eta \in N$ and $h \in H$ we have 
\begin{equation*} 
(g\odot \eta)^{\beta(h)} = (^{\alpha(g)}(\eta ^{\beta(h)})) \star (\eta ^{\beta(h)}) \star (^{\alpha(g)}(e_N \boxdot h)).
\end{equation*}
\end{prop}
\begin{proof}  For all $g\in G,$ $\eta \in N$ and $h \in H$ we obtain
\begin{multline*} 
(g\odot \eta)^{\beta(h)} = (\overline{g\odot \eta}) \star ((g\odot \eta) \boxdot h) \star \overline{(e_N \boxdot h)} =
\overline{(g\odot \eta)} \star (g\odot (\eta \boxdot h)) \star \overline{(e_N \boxdot h)} =\\
\overline{(g\odot \eta)} \star (g\odot (\eta \star (\eta ^{\beta(h)}) \star (e_N \boxdot h))) \star \overline{(e_N \boxdot h)} = \\
\overline{(g\odot \eta)} \star (g\odot \eta) \star \overline{(g \odot e_N )} \star (g\odot ( (\eta ^{\beta(h)}) \star (e_N \boxdot h))) \star \overline{(e_N \boxdot h)} =\\
\overline{(g \odot e_N )} \star (g\odot  (\eta ^{\beta(h)})) \star \overline{(g \odot e_N )} \star (g\odot (e_N \boxdot h)) \star \overline{(e_N \boxdot h)} =\\
(^{\alpha(g)}(\eta ^{\beta(h)})) \star (\eta ^{\beta(h)}) \star (^{\alpha(g)}(e_N \boxdot h)).
\end{multline*}
\end{proof}

\begin{prop}\label{p35}  Let $(G, \cdot ,N, \star   ,\odot )$ be a left bracoid and $(H, \circ )$ be a group. Assume that 
\begin{enumerate}[\rm(a)]
\item $\boxdot $ is a transitive right action of $(H, \circ )$ on $N$  such that
\begin{equation*} 
  g \odot  ( \eta \boxdot  h) = (g \odot   \eta) \boxdot  h  
\end{equation*}
for all $g\in G, h \in H,$ $\eta \in N$, 
\item $^{\alpha(g)}(\eta ^{\beta(h)}) = (^{\alpha(g)}\eta )^{\beta(h)}$ for all $g\in G, h\in H$ and $\eta \in N$.
\end{enumerate}
 Then for all $\eta \in N$ and $h \in H$ we have 
\begin{enumerate}[\rm(1)]
\item $\overline{\eta}^{\beta(h)} = \overline{\eta^{\beta(h)}}$.
\item $\overline{\eta} \boxdot h = (e_N \boxdot h)^2 \star \overline{\eta \boxdot h}$.
\end{enumerate}
\end{prop}
\begin{proof}  (1) Let $\eta \in N$ and $h \in H$. Since $\odot $ is a transitive action of $(G,\cdot  )$ on $N$, there exists $g \in G$ such that 
$g \odot  e_N = \eta$. Hence by Proposition \ref{p34}  we have 
\begin{multline*} 
(^{\alpha(g)}\overline{\eta})^{\beta(h)} = ((^{\alpha(g)}\overline{\eta}) \star \eta \star  \overline{\eta})^{\beta(h)} =
 ((^{\alpha(g)}\overline{\eta}) \star (g \odot  e_N) \star  \overline{\eta}) ^{\beta(h)} =\\
 (g \odot \overline{\eta})^{\beta(h)} = (^{\alpha(g)}(\overline{\eta}^{\beta(h)})) \star (\overline{\eta}^{\beta(h)}) \star (^{\alpha(g)}(e_N\boxdot h)) =\\
   (^{\alpha(g)}(\overline{\eta}^{\beta(h)})) \star (\overline{\eta}^{\beta(h)}) \star 
(\overline{ g \odot   e_N}) \star (g \odot (e_N \boxdot h)) \star (\overline{e_N \boxdot h})  =\\
(^{\alpha(g)}(\overline{\eta}^{\beta(h)})) \star (\overline{\eta}^{\beta(h)}) \star 
\overline{ \eta} \star (\eta \boxdot h) \star (\overline{e_N \boxdot h})  = (^{\alpha(g)}(\overline{\eta}^{\beta(h)})) \star (\overline{\eta}^{\beta(h)}) 
\star ( \eta^{\beta(h)}) .
\end{multline*} 
By ($b$) we have  
$$\overline{\eta}^{\beta(h)} = \overline{\eta^{\beta(h)}}$$
and this implies  (2).
\end{proof}

\begin{thm}\label{thm36}  Let $(G, \cdot ,N, \star   ,\odot )$ be a left bracoid and $(H, \circ )$ be a group. Assume that 
\begin{enumerate}[\rm(a)]
\item $\boxdot $ is a transitive right action of $(H, \circ )$ on $N$  such that
\begin{equation*} 
  g \odot  ( \eta \boxdot  h) = (g \odot   \eta) \boxdot  h  
\end{equation*}
for all $g\in G, h \in H,$ $\eta \in N$, 
\item $^{\alpha(g)}(\eta^{\beta(h)}) = (^{\alpha(g)}\eta )^{\beta(h)}$ for all $g\in G, h\in H$ and $\eta \in N$.
\end{enumerate}
Then $(G, \cdot ,\odot , H, \circ ,\boxdot , N, \star )$ is a two-sided  bracoid.
\end{thm}
\begin{proof}  Let $g\in G, h\in H$ and $\eta\in N$. We have 
\begin{multline*}  
^{\alpha(g)}(\eta^{\beta(h)}) = ^{\alpha(g)}(\overline{\eta} \star (\eta \boxdot h) \star \overline{(e_N \boxdot h)})  =\\
\overline{(g\odot e_N)} \star (g \odot (\overline{\eta} \star (\eta \boxdot h) \star \overline{(e_N \boxdot h)})) \star 
\overline{(\overline{\eta} \star (\eta \boxdot h) \star \overline{(e_N \boxdot h)})}=\\
\overline{(g\odot e_N)} \star (g \odot (\overline{\eta} \star (\eta \boxdot h) \star \overline{(e_N \boxdot h)})) \star 
 \eta \star \overline{(\eta \boxdot h)} \star (e_N \boxdot h)
\end{multline*} 
and
\begin{multline*}  
(^{\alpha(g)}\eta)^{\beta(h)} = (\overline{(g \odot  e_N)} \star (g \odot \eta) \star  \overline{\eta})^{\beta(h)} = \\
\overline{(\overline{(g \odot  e_N)} \star (g \odot \eta) \star  \overline{\eta})} \star ((\overline{(g \odot  e_N)} \star (g \odot \eta) \star  \overline{\eta}) \boxdot h) \star 
\overline{(e_N \boxdot h)} = \\
(g \odot  e_N) \star \overline{(g \odot \eta )} \star  \eta \star ((\overline{(g \odot  e_N)} \star (g \odot \eta) \star  \overline{\eta}) \boxdot h) \star 
\overline{(e_N \boxdot h)} 
\end{multline*}
By ($b$) it follows that
\begin{multline*}
((\overline{(g \odot  e_N)} \star (g \odot \eta) \star  \overline{\eta}) \boxdot h) \star \overline{(g \odot  \eta )} \star (g \odot  e_N)^2 = \\
  (g \odot (\overline{\eta} \star (\eta \boxdot h) \star \overline{(e_N \boxdot h)})) \star 
\overline{(\eta\boxdot h)} \star (e_N \boxdot h)^2.
\end{multline*}
Hence by Lemma \ref{lem23} and Proposition \ref{p35}(2) it follows that
\begin{multline*}
((\overline{(g \odot  e_N)} \star (g \odot \eta) \star  \overline{\eta}) \boxdot h) \star (g \odot  \overline{\eta }) = \\
  (g \odot (\overline{\eta} \star (\eta \boxdot h) \star \overline{(e_N \boxdot h)})) \star (\overline{\eta } \boxdot h).
\end{multline*}

Thus by (\ref{eq2})
\begin{multline*}
((g \odot  (\eta \star (g^{-1} \odot \overline{\eta}) ))\boxdot h) \star (g \odot  \overline{\eta }) = \\
  (g \odot (\overline{\eta} \star (\eta \boxdot h) \star \overline{(e_N \boxdot h)})) \star (\overline{\eta } \boxdot h)
\end{multline*}
and by ($a$) we have
\begin{multline*}
g^{-1} \odot ((g \odot  ((\eta \star (g^{-1} \odot \overline{\eta}) )\boxdot h)) \star (g \odot  \overline{\eta })) = \\
g^{-1} \odot ( (g \odot (\overline{\eta} \star (\eta \boxdot h) \star \overline{(e_N \boxdot h)}))) \star 
(\overline{\eta } \boxdot h)).
\end{multline*}
By applying the action $\odot $ of the group $(G, \cdot )$ on $(N,\star )$ we have 
\begin{multline*}
((\eta \star (g^{-1} \odot \overline{\eta}) )\boxdot h) \star \overline{  \eta } \star \overline{(g^{-1} \odot e_N)} = \\
  \overline{\eta} \star (\eta \boxdot h) \star \overline{(e_N \boxdot h)} \star (g^{-1} \odot (\overline{\eta } \boxdot h)) \star \overline{(g^{-1} \odot e_N)}.
\end{multline*}
Hence by ($a$) we obtain
\begin{equation*}
  (\eta \star  (g^{-1} \odot \overline{\eta})) \boxdot h = 
   (\eta \boxdot h) \star \overline{(e_N \boxdot h)} \star ((g^{-1} \odot \overline{\eta}) \boxdot h ). 
\end{equation*}
Since $\odot $ is a transitive action of $(G, \cdot )$ on $(N, \star )$ for every $\mu \in N$ there exists $g \in G$ such that 
$g^{-1} \odot \overline{\eta} = \mu.$ Therefore  for every $\mu, \eta \in N$ and $h\in H$ we have 
\begin{equation*} 
  (\eta \star   \mu) \boxdot  h = (\eta \boxdot h) \star   \overline{(  e_N \boxdot h)}  \star  (\mu \boxdot h). 
\end{equation*}

\end{proof}

When a two-sided brace $(B,\star ,\cdot)$ is viewed as a left bracoid we have $$a^{\beta (b)} = a\ast b = {^{\alpha (a)}b}$$ for all $a, b \in B$.
Hence Theorem \ref{thm36} is a generalisation of the result of Lau.

\section*{Acknowledgements} 

This research is partially supported by the National Science Centre, Poland, grant no. 2019/35/B/ST1/01115.

\end{document}